\documentclass[12pt]{article}
\usepackage[utf8]{inputenc}
\usepackage{amsthm}
\usepackage{amsmath}
\usepackage{amssymb}
\usepackage{amscd}
\usepackage[all]{xy}
\usepackage[margin=1in]{geometry}
\usepackage{comment}
\usepackage{inputenc}
\usepackage{pdfpages}
\usepackage{tikz}
\usepackage{pgfplots}
\usepackage{latexsym}
\usepackage{tikz-cd}
\usetikzlibrary{patterns}

\begin{document}

\title{The reciprocal complements of classes of integral domains}
\author{Lorenzo Guerrieri \thanks{ Jagiellonian University, Instytut Matematyki, 30-348 Krak\'{o}w \textbf{Email address:} lorenzo.guerrieri@uj.edu.pl }  
}
\maketitle

\begin{abstract}

\noindent Given an integral domain $D$ with quotient field $\mathcal{Q}(D)$, the reciprocal complement of $D$ is the subring $R(D)$ of $\mathcal{Q}(D)$ whose elements are all the sums $\frac{1}{d_1}+\ldots+\frac{1}{d_n} $ for $d_1, \ldots, d_n$ nonzero elements of $D$.
In this article we study problems related with prime ideals, localizations and Krull dimension of rings of the form $R(D)$ and we describe the reciprocal complements of classes of domains, including semigroup algebras and $ D+ \mathfrak{m} $ constructions. We also characterize when $R(D)$ is a DVR.
\medskip

\end{abstract}

\newtheorem{thm}{Theorem}[section]
\newtheorem{lemma}[thm]{Lemma}
\newtheorem{prop}[thm]{Proposition}
\newtheorem{cor}[thm]{Corollary}
\newtheorem{problem}[thm]{Problem}
\newtheorem{construction}[thm]{Construction}

\theoremstyle{definition}
\newtheorem{defi}[thm]{Definitions}
\newtheorem{definition}[thm]{Definition}
\newtheorem{rem}[thm]{Remark}
\newtheorem{example}[thm]{Example}
\newtheorem{question}[thm]{Question}
\newtheorem{conjecture}[thm]{Conjecture}
\newtheorem{comments}[thm]{Comments}

\newtheorem{discussion}[thm]{Discussion}

\newcommand{\N}{\mathbb{N}}
\newcommand{\m}{\mathfrak{m}}
\newcommand{\p}{\mathfrak{p}}
\newcommand{\q}{\mathfrak{q}}
\newcommand{\Z}{\mathbb{Z}}
\newcommand{\Q}{\mathbb{Q}}
\newcommand{\al}{\boldsymbol{\alpha}}
\newcommand{\be}{\boldsymbol{\beta}}
\newcommand{\de}{\boldsymbol{\delta}}
\newcommand{\e}{\textbf{e}}
\newcommand{\om}{\boldsymbol{\omega}}
\newcommand{\g}{\boldsymbol{\gamma}}
\newcommand{\te}{\boldsymbol{\theta}}
\newcommand{\he}{\boldsymbol{\eta}}
\def\min{\mbox{\rm min}}
\def\max{\mbox{\rm max}}
\def\ff{\frak}
\def\Spec{\mbox{\rm Spec }}
\def\Zar{\mbox{\rm Zar }}
\def\Proj{\mbox{\rm Proj }}
\def\hgt{\mbox{\rm ht }}
\def\type{\mbox{ type}}
\def\Hom{\mbox{ Hom}}
\def\rank{\mbox{ rank}}
\def\Ext{\mbox{ Ext}}
\def\Ker{\mbox{ Ker}}
\def\Max{\mbox{\rm Max}}
\def\End{\mbox{\rm End}}
\def\xpd{\mbox{\rm xpd}}
\def\Ass{\mbox{\rm Ass}}
\def\emdim{\mbox{\rm emdim}}
\def\epd{\mbox{\rm epd}}
\def\repd{\mbox{\rm rpd}}
\def\ord{\mbox{\rm ord}}
\def\gcd{\mbox{\rm gcd}}
\def\Tr{\mbox{\rm Tr}}
\def\Res{\mbox{\rm Res}}
\def\Ap{\mbox{\rm Ap}}
\def\sdefect{\mbox{\rm sdefect}}
\maketitle

\section{Introduction}

Given an integral domain $D$ with quotient field $\mathcal{Q}(D)$, the reciprocal complement of $D$ is the subring $R(D)$ of $\mathcal{Q}(D)$ generated by all the elements $\frac{1}{d} $ for all nonzero elements $d \in D$. 
Thus, the elements of $R(D)$ are all the sums $\frac{1}{d_1}+\ldots+\frac{1}{d_n} $ for $d_1, \ldots, d_n$ nonzero elements of $D$.
This ring construction has been introduced in \cite{Eps3}, motivated by the study of Egyptian fractions in arbitrary integral domains.

Classically, an Egyptian fraction is a representation of a rational number as a sum of distinct unit fractions. The fact that any rational number admits representations as an Egyptian fraction was probably known already to ancient Egyptians, and for sure was known to Fibonacci, who wrote a proof of this fact in his book Liber Abaci (see \cite{fibonacci}). In modern time, Egyptian fractions have been deeply studied in the context of number theory. 

The concept of Egyptian fraction has been extended from the setting of integers and rational numbers to arbitrary integral domains in \cite{GLO} and further studied in \cite{Eps1}, \cite{Eps2}, \cite{Eps3}, \cite{Epsmonthly}. In \cite{Eps2} the same notion has been generalized also to the setting of rings with zerodivisors.

A integral domain $D$ is defined to be \it Egyptian \rm if every nonzero element of $D$ can be expressed as a sum of distinct unit fractions in $D$, i.e. as a sum $\frac{1}{d_1}+\ldots+\frac{1}{d_n} $ for $d_1, \ldots, d_n$ distinct nonzero elements of $D$ (however, in \cite[Theorem 2]{GLO} it is shown that the assumption of having distinct denominators can be removed).

Many integral domains are Egyptian: for instance the ring of integer $\Z$, any semilocal domain, any overring and any integral extension of an Egyptian domain. The easiest example of a non-Egyptian domain is a polynomial ring $A[X]$ where $A$ is any integral domain and $X$ is an indeterminate over $A$. In this case every non-constant polynomial cannot be represented as a sum of reciprocals of elements in $A[X]$ (for all these results see \cite{GLO}). 

Hence, it turned out that, even if $\Z$ and $K[X]$ (where $K$ is any field) are both very well-behaved Euclidean domains, the first one is Egyptian and the second is not. To understand better this difference, the reciprocal complement $R(D)$ of a domain $D$ has been defined in \cite{Eps3}, and it has been immediately observed that an integral domain $D$ is Egyptian if and only if $R(D)$ coincides with the whole quotient field of $D$.

Therefore, the ring $R(D)$ can be used to measure how far a domain is from being Egyptian. While clearly $R(\Z)=\Q$, it has been shown in \cite{Eps3} that $R(K[X])= K[X^{-1}]_{(X^{-1})}$ and, more in general, the reciprocal complement of any Euclidean domain is either a field or a DVR.

Given now any non-Egyptian domain $D$ (more classes of them have been described in \cite{GLO}, \cite{Eps1}), it is an interesting problem to describe the properties of the ring $R(D)$ from the point of view of multiplicative ideal theory.  

This problem has been considered in \cite{EGL} in the case $D=K[X_1, \ldots, X_n]$ is a polynomial ring in several variables over a field. If $n \geq 2$, the results obtained are quite surprising compared to the case $n=1$.

Setting $R_n= R(K[X_1, \ldots, X_n])$, we have that $R_n$ has some of the features of the localized polynomial ring $ K[X_1, \ldots, X_n]_{(X_1, \ldots, X_n)} $. Indeed, $R_n$ is a local domain of Krull dimension $n$ with infinitely many prime ideals of every height $i=1, \ldots, n-1$. It has also special prime ideals $\q_i$ for $i < n$ such that 
$  \frac{R_n}{\q_i} \cong R_i, $ and $ (R_n)_{\q_i}= R(K(X_1, \ldots, X_i)[X_{i+1}, \ldots, X_n]). $ Moreover, still for $i < n$, it satisfies the equality $R_n \cap K(X_1, \ldots, X_i)= R_i$.

However, $R_n$ presents also a very different behavior with respect to a regular local ring. For instance, the element  $\frac{1}{X_1X_2\cdots X_n}$ is contained in every nonzero prime ideal of $R_n$. This makes $R_n$ non-Noetherian if $n\geq 2$, and furthermore, in the case $n\geq 2$, it has also been proved that $R_n$ is not integrally closed. 

All these results are a good motivation to continue the study of the ring $R(D)$ as a general ring construction with the aim of understanding its structure for more classes of integral domains.

In Section 2 of this article, we first focus on general results related with the prime spectrum of $R(D)$ for arbitrary integral domain $D$. We prove in Proposition \ref{primeideals} that for every nonzero element $x \in D$, there exists a unique prime ideal $\p_x$ of $R(D)$, maximal with respect to the property of excluding the element $\frac{1}{x}$. Conversely, if $R(D)$ has finite Krull dimension, we prove in Theorem \ref{pseudoradical} that every prime ideal of $R(D)$ is of the form $\p_x$ for some $x \in D$. In particular, in this case every localization of $R(D)$ at a prime ideal
is equal to $R(S^{-1}D)$ for some multiplicatively closed subset $S$ of $D$, and $R(D)$ has nonzero pseudoradical (the pseudoradical of a domain is the intersection of all nonzero prime ideals).

A similar result is obtained in Theorem \ref{quotients} for the quotients of $R(D)$ by prime ideals. It is shown that any quotient $R(D)/\p$ is equal to $R(L)$ for some subring $L \subseteq D$.

Along the way, regarding the open problem of characterizing when $R(D)$ is Noetherian, it is shown in Corollary \ref{noeth1} that, if $R(D)$ is Noetherian, then its Krull dimension cannot exceed one. 

Another problem that we consider in this article is the computation of the Krull dimension of $R(D)$. The following general conjecture was proposed in \cite[Question 5]{EGL}:
\begin{conjecture} 
\label{conj1}
Let $D$ be an integral domain and let $R(D)$ be its reciprocal complement. Then
 $\dim (R(D)) \leq \dim (D)$.
\end{conjecture}

In Section 3, we give a positive answer to this conjecture in the case where $D$ is a finitely generated algebra over a field (see Theorem \ref{fgalgebras}). Moreover, we use the same method adopted in the proof of this result to construct an example of a principal ideal domain $D$ such that $R(D)$ is not integrally closed, hence not a DVR (and not isomorphic to an overring of $D$), showing that the results true for Euclidean domains can fail already in the class of PIDs.

In Section 4, we study the reciprocal complements of semigroup algebras defined over a field, focusing on the case where the defining semigroup $S$ is contained in the positive part of a totally ordered abelian group $G$ of finite rank. In this case we prove that $R(K[S])$ is isomorphic to an overring of the localized semigroup algebra (as already known in the specific case of polynomial rings) and, in the case where $S$ is equal to the positive part of $G$, the reciprocal complement $R(K[S])$ is a valuation domain, as in the case of the polynomial ring $K[X]$ (see Theorem \ref{totalordervaluation}). 
In Theorem \ref{numerical}, we describe explicitly $R(K[S])$ in the case when the positive part of $G$ is the integral closure of $S$ and $K$ has characteristic zero. This covers the particular case where $S \subseteq \N$ is a numerical semigroup. 
We also prove in Proposition \ref{noethsem} that $R(K[S])$ is Noetherian if and only if $K[S]$ is Noetherian and one dimensional,  and finally we prove that Conjecture \ref{conj1} holds true for semigroup algebras of our kind, and we compute the Krull dimension of $R(K[S])$ in some relevant cases in Theorem \ref{dimth1}.

In Section 5, we study the reciprocal complements of rings of the form $D+\m$, where $\m$ is a prime ideal of an overring of $D$ and, applying the obtained results to semigroup algebras, we construct in Theorem \ref{dimth2}, for every $n > m$, an integral domain $D$ such that $\dim(D)=n$ and $\dim(R(D))=m$.

Finally, in Section 6, we study the case when $R(D)$ is a DVR and we prove that in this case, after localizing $D$ by inverting all its Egyptian elements, we obtain an integral domain isomorphic to a polynomial ring in one variable over a field. This provides a partial converse to the main theorem of \cite{Eps3}.

\section{Prime ideals of reciprocal complements}

In the following $D$ will denote an integral domain with quotient field $\mathcal{Q}(D)$. By overring of $D$ we mean as usual a ring lying between $D$ and $\mathcal{Q}(D)$.

A nonzero element $x \in \mathcal{Q}(D)$ is \it $D$-Egyptian \rm if there exist nonzero distinct elements $d_1, \ldots, d_n \in D$ such that $$ x= \dfrac{1}{d_1} + \ldots + \dfrac{1}{d_n}. $$ Notice that by \cite[Theorem 2]{GLO} the assumption of $d_1, \ldots, d_n$ being distinct is in fact redundant in the context of integral domains. When there is no risk of confusion we omit the prefix $D$- and simply talk about Egyptian elements. 

An integral domain $D$ is said to be \it Egyptian \rm if every nonzero element of $D$ is Egyptian. The \it reciprocal complement (or ring of reciprocal complements) \rm of $D$ is defined as the subring $R(D)$ of $\mathcal{Q}(D)$ generated by all the elements $\frac{1}{d}$ for $d \in D \setminus \lbrace 0 \rbrace.$

The immediate connection between Egyptian domains and reciprocal complements stands in the fact that an integral domain $D$ is Egyptian if and only if $R(D) = \mathcal{Q}(D)$. \\

In this section we investigate properties of prime ideals, localizations and quotients of reciprocal complements for arbitrary integral domains.

We start by recalling the following main result from \cite{EGL}.

\begin{thm} \rm \cite[Theorem 2.6]{EGL} \it
 \label{islocal!!}
 Let $D$ be an integral domain. Then, the ring $R(D)$ is local and its maximal ideal is generated by all the elements $\frac{1}{d}$ where $d \in D$ is not an Egyptian element.
 \end{thm}


The next basic lemma studies the behavior of the reciprocal complement with respect to the localizations of $D$.

\begin{lemma}
\label{localizations}
Let $D$ be an integral domain and $S$ a multiplicatively closed subset of $D$. Then $R(S^{-1}D)= R(D)[S]$.
\end{lemma}

\begin{proof}
The ring $ R(S^{-1}D) $ is generated by all the elements of the form $\frac{s}{d}$ for $s \in S$ and $d \in D$. Any of these elements is clearly in $R(D)[S]$. For the opposite containment, we recall that taking ring of reciprocal complements preserves inclusions and therefore $R(D) \subseteq R(S^{-1}D)$. Also $S \subseteq R(S^{-1}D)$ since $s^{-1} \in S^{-1}D$ for every $s \in S$. This concludes the proof.
\end{proof}

Recall now that all the units of $D$ are trivially Egyptian elements, and if $x,y$ are Egyptian elements of $D$, then $x+y$ and $xy$ are also Egyptian elements (see \cite[Proposition 8]{GLO}).
 Let us denote by $E= E(D)$ the set of Egyptian elements of $D$ and by $\mathcal{E} = \mathcal{E}(D)$ the ring $E \cup \lbrace 0 \rbrace$. Denote the quotient field of $\mathcal{E}$ by $\mathcal{Q}(\mathcal{E})$.
 The set $E$ is multiplicatively closed, thus we can consider the localization $E^{-1}D$. We recall the following properties:
 
 \begin{thm} \rm \cite[Proposition 2.2]{EGL} \it
 \label{reductiontoKalgebra}
 The following assertions hold:
 \begin{enumerate}
 \item[(1)] $R(E^{-1}D)= R(D)$.
 \item[(2)] The set of Egyptian elements of $E^{-1}D$ coincides with the field $\mathcal{Q}(\mathcal{E})$.
 \end{enumerate}
 \end{thm}
 


 An interesting corollary of the previous results comes as follows:

\begin{cor}
\label{corlocalizations}
Let $D$ be an integral domain and let $E$ denote the subset of Egyptian elements of $D$. The following conditions are equivalent:
 \begin{enumerate}
 \item[(1)] $D$ is Egyptian.
 \item[(2)]  $E \cap \p \neq \emptyset$ for all nonzero prime ideals $\p$ of $D$.
 \item[(3)]  $E \cap \p \neq \emptyset$ for all but finitely many prime ideals $\p$ of $D$.
 \end{enumerate}
\end{cor}

\begin{proof}
If $D$ is Egyptian, then $D = E \cup \lbrace 0 \rbrace$ and condition (2) follows immediately. Also the implication (2) $\Rightarrow$ (3) is obvious. Assuming (3), we can pass to the localization $E^{-1}D$ and use that $R(E^{-1}D)=R(D)$ by Theorem \ref{reductiontoKalgebra}. Since all but finitely many primes of $D$ meet $E$, we obtain that $E^{-1}D$ is semilocal, hence Egyptian by \cite[Corollary 1]{GLO}. It follows that $R(D)= R(E^{-1}D) = \mathcal{Q}(D)$ and thus $D$ is Egyptian.
\end{proof}

Thanks to Theorem \ref{reductiontoKalgebra}, in the rest of the paper we will often assume that $D$ is an integral domain containing a field $K$ and the only Egyptian elements of $D$ are the nonzero elements of $K$. Equivalently, $D$ is an integral domain such that every Egyptian element is a unit.


We can specialize the result of Lemma \ref{localizations} to the case of multiplicatively closed sets generated by one element. Before doing this we recall a result about algebraic elements.

\begin{lemma} \rm \cite[Proposition 5]{GLO} \it
\label{algegyp}
Let $D$ be a $K$-algebra and let $x \in D$ be a nonzero element. If $x$ is algebraic over $K$ then it is an Egyptian element of $D$.
\end{lemma}

\begin{proof}
Since $x$ is algebraic over $K$, there exist $a_0, \ldots, a_c \in K$ such that $\sum_{i=0}^c a_i x^i = 0$ and $c$ is the minimal integer such that such algebraic equation exists. Hence, $x( \sum_{i=1}^c a_i x^{i-1}) \in K$ and $ \sum_{i=1}^c a_i x^{i-1} \neq 0$ by assumption on $c$. The thesis follows since $x$ is a unit and units are Egyptian elements. 
\end{proof}

\begin{cor}
\label{cor2loc} Let $D$ be an integral domain containing a field $K$. Consider the multiplicatively closed set $S = K[x] \setminus \lbrace  0 \rbrace$ for some $x \in D \setminus K$. Then $$ R(S^{-1}D)=R(D)[S]=R(D)[x]= R(D[x^{-1}]).$$
\end{cor}

\begin{proof}
The first and last equality follow by Lemma \ref{localizations}. For the equality in the middle, we observe first that $R(D)[S] \supseteq R(D)[x]$ since $x \in S$. For the other containment, if $x$ is algebraic over $K$, then it is Egyptian by Lemma \ref{algegyp} and all the elements of $K[x]$ are also Egyptian. In this case $R(D)[S]=R(D)=R(D)[x]$.

If $x$ is not algebraic over $K$, we identify $K[x]$ with a polynomial ring in one variable over $K$ and recall that by \cite[Proposition 7]{GLO}, $R(K[x][x^{-1}])= K(x)$. It follows that $K(x) \subseteq R(D[x^{-1}]) = R(D)[x]$ and therefore $S \subseteq R(D)[x]$.
\end{proof}

The fact that $R(D)$ is always local has some nice consequences related to the behavior of its prime ideals. Item 3 of the following proposition has been already observed in \cite[Propositions 2.10-2.11]{EGL}. We prove it again for sake of completeness.

  \begin{prop}
\label{primeideals}
Let $D$ be an integral domain. The following properties hold for the ring $R(D)$:
\begin{enumerate}
\item[(1)] For every nonzero $x \in D$ there exists a unique prime ideal of $R(D)$ maximal with respect to the property of excluding the fraction $ \frac{1}{x} $.
\item[(2)] For $x,y \in D$ (such that $x+y \neq 0$), if $\frac{1}{x+y} \in \p $ for some prime ideal $\p$ of $R(D)$, then at least one of $\frac{1}{x}$ and $\frac{1}{y}$ is in $\p$.
\item[(3)] Given $x_1, \ldots, x_n \in D$, suppose that $\sum_{i=1}^n \frac{1}{x_i}$ cannot be expressed as a sum of less than $n$ reciprocals of elements of $D$.
Then, if $ \sum_{i=1}^n \frac{1}{x_i} \in \p $ for some prime ideal $\p$ of $R(D)$, it follows that $\frac{ 1}{x_i} \in \p$ for every $i$. In particular every prime ideal of $R(D)$ can be generated by reciprocals of elements of $D$.
\end{enumerate}
\end{prop}

\begin{proof}
Let us prove (1). If $x $ is Egyptian, then it is a unit in $R(D)$ and all the prime ideals of $R(D)$ exclude $\frac{1}{x}$. In this case the thesis follows since $R(D)$ is local 
by Theorem \ref{islocal!!}.

 Suppose that $x $ is not Egyptian. Thus $x \not \in R(D)$.
Observe that by Corollary \ref{cor2loc}, $R(D)[x]= R(D[x^{-1}])$ is a proper overring of $R(D)$. But $R(D[x^{-1}])$ is a reciprocal complement, hence local by Theorem \ref{islocal!!}. A local localization of an integral domain has to be the localization at a prime ideal. Thus $R(D)[x] = R(D)_{\p}$ for some prime ideal $\p$.
By standard properties of localizations, $\p$ is the unique prime ideal maximal with respect to the property of excluding $ \frac{1}{x} $.  

To prove (2), write $$  \frac{1}{x+y} \left( \frac{1}{x} + \frac{1}{y} \right) = \frac{1}{x+y} \frac{x+y}{xy} = \frac{1}{xy}. $$  This shows that $ \frac{1}{xy} \in \p$ if $\frac{1}{x+y} \in \p $. The thesis is now an immediate consequence of the definition of prime ideal.  \\
To prove (3), set $x= x_1 \cdots x_n$
and observe that $ (\sum_{i=1}^n \frac{x}{x_i} )^{-1}  $ is a well defined element of $R(D)$ since we assumed $\sum_{i=1}^n \frac{1}{x_i} \neq 0$.
Hence, if $ \sum_{i=1}^n \frac{1}{x_i} \in \p $, we have
$$ \dfrac{1}{x} = \left( \sum_{i=1}^n \frac{1}{x_i} \right) \left( \sum_{i=1}^n \frac{x}{x_i} \right)^{-1} \in \p. $$
It follows that $ \frac{1}{x_j} \in \p $ for some $j$. Subtracting this term from $\sum_{i=1}^n \frac{1}{x_i}$ and iterating the procedure we get the thesis.
\end{proof}

As a consequence we can prove that all the reciprocal complements that have Krull dimension 2 satisfy a property which has been proved specifically for $R(K[X,Y])$ in \cite[Theorem 7.3]{EGL}. The same property does not hold anymore in larger dimension.

\begin{cor}
\label{dim2}
Suppose that $R(D)$ has Krull dimension 2. Then every finitely generated proper ideal of $R(D)$ is contained in all but finitely many prime ideals.
\end{cor}

\begin{proof}
By Theorem \ref{reductiontoKalgebra}, we assume that $D$ is a $K$-algebra whose only Egyptian elements are the nonzero elements of $K$. Notice that it is sufficient to prove that every element of the form $ \frac{1}{x} $ for $x \in D \setminus K$ is in all but finitely many primes. Indeed any non-unit in $R(D)$ is a finite sum of such elements. 

By item 1 of Proposition \ref{primeideals}, there exists a unique prime ideal maximal with respect to the property of excluding $\frac{1}{x}$. Since $x \not \in K$, this prime cannot be the maximal ideal of $R(D)$. Hence it must be a prime ideal of height 1 or 0. By uniqueness it follows that $\frac{1}{x} $ is excluded by at most two primes (including the zero prime).
\end{proof}

We know that in some case $R(D)$ can be Noetherian, for instance if $D$ is Egyptian or if $D=K[X]$. We can prove now that in the case $R(D)$ is Noetherian, its Krull dimension can not exceed 1.

\begin{cor}
\label{noeth1}
Suppose that $R(D)$ is Noetherian. Then $\dim(R(D)) \leq 1$.
\end{cor}

\begin{proof}
Since $R(D)$ is Noetherian and local, it must have finite Krull dimension.
Suppose by way of contradiction $\dim (R(D)) \geq 2$. Hence $D$ is not Egyptian and we can find an  element $x \in D$ such that $ \frac{1}{x} $ is not a unit in $ R(D)$. By Proposition \ref{primeideals}, there exists a unique prime ideal $\p$ maximal with respect to the property of excluding $\frac{1}{x}$. By the theory of primary decomposition in Noetherian rings and Krull's Altitude Theorem, $\frac{1}{x}$ is contained in only finitely many height one primes. Since $\dim (R(D)) \geq 2$, there are infinitely many height one primes. By uniqueness of the property of $\p$, we must have $\hgt \p > 1$. Moreover, $\p$ cannot be the maximal ideal since $\frac{1}{x}$ is not a unit. If $\dim (R(D)) = 2$ we get an immediate contradiction. 

If $\dim (R(D)) > 2$ we can pass to the ring $R(D[x^{-1}])= R(D)[x]= R(D)_{\p}$ which is still Noetherian and $ 1 < \dim (R(D)_{\p}) < \dim (R(D)) $. By induction on the dimension of $ R(D)$ we find a contradiction.
\end{proof}

The \it pseudoradical \rm of an integral domain is the intersection of all the nonzero prime ideals \cite{gilmerpseudo}, \cite{Kap}. An integral domain with nonzero pseudoradical is called a \it G-domain. \rm
A standard fact is that, if $x $ is a nonzero element of the pseudoradical of $D$, then $D[x^{-1}]= \mathcal{Q}(D)$.

In \cite[Proposition 2.12]{EGL} it is proved that the ring of reciprocal complements of the polynomial ring $K[X_1, \ldots, X_n]$ (and more in general of every finitely generated $K$-algebra) has nonzero pseudoradical. 

We prove that this fact is true more in general in the case when $R(D)$ has finite Krull dimension. Together with this we prove a stronger result: if $R(D)$ has finite Krull dimension, then all its localizations at prime ideals are reciprocal complements of some localization of $D$.

\begin{definition}
\label{defpx}
By item 1 of Proposition \ref{primeideals}, given a nonzero $x \in D$, there exists a unique prime ideal of $R(D)$ maximal with respect to the property of excluding $\frac{1}{x}$. We denote this prime ideal by $\p_x$.
\end{definition}
 
 We have the following:

\begin{lemma}
\label{pxpy}
Given nonzero $x,y \in D$, 
then the ideal $\p_{xy}$ is the unique largest prime ideal of $R(D)$ contained in $\p_x \cap \p_y$.
\end{lemma}

\begin{proof}
By definition of $\p_{xy}$ in terms of its maximality property we have that $\p_{xy} \subseteq \p_x \cap \p_y $. But also, if $\q$ is a prime ideal of $R(D)$ such that $ \q \subseteq \p_x \cap \p_y$ we must have $\frac{1}{xy} \not \in \q$, hence $\q \subseteq \p_{xy}$.
\end{proof}

We recall that two prime ideals $\p, \q$ are \it adjacent \rm if $\p \subseteq \q$ and there are no other prime ideals properly contained between them.

\begin{thm}
 \label{pseudoradical}
 Let $D$ be an integral domain such that $R(D)$ has finite Krull dimension and let $\p$ be a prime ideal of $R(D)$. Then there exists $x \in D$ such that $\p=\p_x$. In particular, $R(D)$ has nonzero pseudoradical.  
 \end{thm}
 
 \begin{proof}
By Theorem \ref{islocal!!}, $R(D)$ has a unique maximal ideal which is equal to $\p_u$ for any unit $u \in D$. Therefore, we can work by decreasing induction on the height of the primes and, fixed a prime ideal $\p$, we can assume that the thesis holds true for every prime ideal properly containing $\p$. We have now two cases: \\
\bf Case 1: \rm there exist two incomparable prime ideals $\q_1, \q_2$ adjacent to $\p$ such that $\p \subseteq \q_1 \cap \q_2$. \\
Since they are incomparable we must have $\q_1 \cap \q_2 \subsetneq \q_i$ for $i=1,2$.
In this case, by inductive hypothesis we have $\q_1= \p_{x}$, $\q_2= \p_{y}$ for some $x,y \in D$ and, by adjacency and
by Lemma \ref{pxpy} we must have $\p = \p_{xy}$. \\ 
\bf Case 2: \rm there exists a unique minimal prime ideal $\q$ adjacent to $\p$ such that $\p \subsetneq \q$. \\
Also in this case we suppose by induction that $\q = \p_x$. By item 3 of Proposition \ref{primeideals}, $\q$ can be generated by elements of the form $\frac{1}{d}$ with $d \in D$. Thus there exists $y  \in D$ such that $ \frac{1}{y} \in \q \setminus \p $. This implies that $\frac{1}{xy} \not \in \p$ and $\p \subseteq \p_{xy} \subseteq \p_x$. Since $ \frac{1}{y} \in \q $, the adjacency of $\p $ and $\q$ implies that $\p = \p_{xy}$.

Iterating this argument, since $R(D)$ has finite Krull dimension, we obtain that every prime of $R(D)$ is of the form $\p_x$ and in particular there exists a nonzero $x \in D$ such that $(0)R(D)=\p_x$. Hence, $\frac{1}{x}$ is a nonzero element of the pseudoradical of $R(D)$.
 \end{proof}

\begin{cor}
\label{localizcor}
Let $D$ be an integral domain such that $R(D)$ has finite Krull dimension.
Then every localization of $R(D)$ at a prime ideal is the reciprocal complement of some localization of $D$.
\end{cor}

\begin{proof}
Observe that by Corollary \ref{cor2loc} and Proposition \ref{primeideals} we have $R(D)_{\p_x}= R(D)[x]=R(D[x^{-1}])$. Apply then Theorem \ref{pseudoradical}.
\end{proof}

In \cite[Theorem 146]{Kap}, it is shown that the Krull dimension of a Noetherian G-domain is at least one. Since Noetherian local domains have finite Krull dimension, this fact combined with Theorem \ref{pseudoradical} gives another proof of Corollary \ref{noeth1}.  

\begin{rem}
\label{locallyegyptian}
In the terminology of \cite[Definition 3.4]{Eps1}, an integral domain $D$ is \it generically Egyptian \rm if there exists some $x \in D$ such that $D[x^{-1}]$ is Egyptian. If $R(D)$ has finite Krull dimension, then it has nonzero pseudoradical. In particular there exists $x \in D$ such that $\p_x = (0)$, therefore $\frac{1}{x}$ is in the pseudoradical of $R(D)$ and $R(D[x^{-1}])= R(D)[x] = \mathcal{Q}(D)$. It follows that $D[x^{-1}]$ is Egyptian and $D$ is generically Egyptian.

Conversely, if $D$ is generically Egyptian, then there exists $x \in D$ such that $D[x^{-1}]$ is Egyptian.
It easily follows that $R(D)[x] = \mathcal{Q}(D)$ and $R(D)$ has nonzero pseudoradical.
This proves that $R(D)$ is a G-domain if and only if $D$ is generically Egyptian.

If Conjecture \ref{conj1} were true, it would follow by Theorem \ref{pseudoradical} that every integral domain of finite Krull dimension is generically Egyptian.
\end{rem}

\begin{rem}
\label{infinitedimension}
The results of Theorem \ref{pseudoradical} may not be true for rings of reciprocal complements of infinite Krull dimension. 
In \cite[Example 3.7]{Eps1} is provided an example of an integral domain that is not generically Egyptian. 

Let $D=K[X_1, X_2, \ldots]$ be the polynomial ring in infinitely many variables over a field $K$. The ring $R(D)$ is local and, for every $n$, there exists a prime ideal $\q_n$ such that $R(D)_{\q_n}$ is equal to the reciprocal complement of a polynomial ring in $n$ variables over a field. Indeed, by Corollary \ref{cor2loc}, 
$$R(D)[X_{n+1}, X_{n+2}, \ldots] = R(D[X_{n+1}^{-1}, X_{n+2}^{-1}, \ldots]) = R(K(X_{n+1}, X_{n+2}, \ldots)[X_1, \ldots, X_n]). $$
It follows that $R(D)[X_{n+1}, X_{n+2}, \ldots]$ is local and therefore equal to the localization of $R(D)$ at a prime ideal $\q_n$.
By \cite[Theorem 4.4]{EGL}, $R(D)_{\q_n}$ has Krull dimension $n$. Thus
$R(D)$ must have infinite Krull dimension.

We show that the zero ideal of $R(D)$ is not equal to $\p_z$ for every $z \in D$. To get this, suppose $(0)=\p_z$. This would imply that the quotient field of $D$ is equal to $R(D)[z] = R(D[\frac{1}{z}])$. Hence, $D[\frac{1}{z}]$ is Egyptian, but this is a contradiction. Indeed,
we can find a sufficiently large $n $ such that $z \in K[X_1, \ldots, X_n]$ and $D[\frac{1}{z}] = D'[X_{n+1}, X_{n+2}, \ldots]$ is a polynomial ring over the integral domain $D'= K[X_1, \ldots, X_n][\frac{1}{z}]$, hence it cannot be Egyptian. 
\end{rem}

We have proved that, in the finite dimensional case, localizations at prime ideals of a reciprocal complement are still reciprocal complements. We prove now the same result for quotients at prime ideals, this time without assumptions on the Krull dimension.

\begin{thm}
 \label{quotients}
 Let $D$ be an integral domain and let $\p$ be a prime ideal of $R(D)$. Then there exists a subring $L \subseteq D$ such that 
 $ \frac{R(D)}{\p} \cong R(L). $  
 \end{thm}
 
 \begin{proof}
Set $L:=\lbrace 0 \rbrace \cup \lbrace x \in D \, : \, \frac{1}{x} \not \in \p \rbrace$. We prove that $L$ is a subring of $D$. 

All the Egyptian elements of $D$ (in particular all the units) are also elements of $L$. Thus, $L = \lbrace 0 \rbrace \cup E $ if and only if $\p$ is the maximal ideal of $R(D)$. In this case the result follows by Theorem \ref{reductiontoKalgebra}.
Suppose then that $\p$ is not maximal. In this case we need to prove that $L$ is closed under sums and products. 

Pick nonzero elements $x,y \in L$. We can further assume $x,y$ to be not associated in $D$ since the units of $D$ are also elements of $L$ (in particular $x+y \neq 0$).
By definition of $L$, $ \frac{1}{x}, \frac{1}{y} \not \in \p $. This immediately implies that $ \frac{1}{xy} \not \in \p $, hence $xy \in L$.
By item 2 of Proposition \ref{primeideals}, this also implies that $ \frac{1}{x+y} \not \in \p $, hence $x+y \in L$.

Let now $\pi$ be the canonical quotient map $R(D) \twoheadrightarrow  \frac{R(D)}{\p}$ and let $\pi'$ be its restriction to the subring $R(L) \subseteq R(D)$. By definition of $L$, $\pi'$ is still surjective. 
To prove injectivity pick a sum of reciprocals $\alpha= \frac{1}{x_1}+\ldots + \frac{1}{x_n} $ with $x_1, \ldots, x_n \in L$ and suppose without loss of generality that that it cannot be rewritten as a sum of less than $n \geq 1$ reciprocals (in particular $\alpha \neq 0$). If $\pi'(\alpha)=0$, we would have $\alpha \in \p$. But
by item 3 of Proposition \ref{primeideals} this forces $\frac{1}{x_1},\ldots, \frac{1}{x_n} \in \p $ contradicting the fact that $x_1, \ldots, x_n \in L$. It follows that $\pi'$ is injective and therefore an isomorphism.
 \end{proof}

\section{Finitely generated $K$-algebras and Krull dimension of reciprocal complements}

In this section we discuss some result related with the Krull dimension of reciprocal complements. In Conjecture \ref{conj1}, we conjectured that for any integral domain $D$, the inequality $\dim(R(D)) \leq \dim(D)$ holds. We can prove this inequality for finitely generated algebras over a field $K$. Thanks to Theorem \ref{reductiontoKalgebra} it is not restrictive to work with $K$-algebras. In the case the algebra is also finitely generated, the main tool that we use in the proof is the well-known Zariski's Lemma (see \cite[Proposition 7.9]{AM} or \cite{zariski}), stating that the residue fields at maximal ideals of a finitely generated $K$-algebra are (isomorphic to) algebraic extensions of the base field $K$.

\begin{lemma}
\label{lemfgalg}
Let $D = K[f_1, \ldots, f_n]$ be a finitely generated $K$-algebra. Let $S = K[x] \setminus \lbrace 0 \rbrace$ for some $x \in D$. Then $S^{-1}D$ is a finitely generated $K(x)$-algebra.
\end{lemma}

\begin{proof}
Clearly $S^{-1}D$ contains the field $K(x)$ and therefore is a $K(x)$-algebra.
A general element of $S^{-1}D$ has the form $\phi= \frac{f}{g}$ with $f \in D$ and $g \in K[x]$. Writing $f = \sum_{i=1}^n a_i f_i$ with $a_i \in K$, we obtain $ \phi= \sum_{i=1}^n \frac{a_i}{g} f_i $ where $ \frac{a_i}{g} \in K(x). $ 
\end{proof}

\begin{thm}
 \label{fgalgebras}
 Let $D $ be a finitely generated $K$-algebra. Then $\dim(R(D)) \leq \dim(D)$.
 \end{thm}
 
 \begin{proof}
Let us work by induction on $n=\dim(D)$. If $n=0$, $D$ is a field and $D=R(D)$. Hence, suppose by way of contradiction $\dim(R(D)) > \dim(D)=n > 0$.
By these assumptions, there exists a non-maximal prime ideal $\q$ of $R(D)$ such that $\hgt \q = n$. Calling $\m$ the maximal ideal of $R(D)$, we can find an element $x \in D$ such that $\frac{1}{x} \in \m \setminus \q$ (recall that by item 3 of Proposition \ref{primeideals}, every prime ideal of $R(D)$ is generated by reciprocals of elements of $D$). By Definition \ref{defpx} we have $\q \subseteq \p_x \subsetneq \m$. 

Let $S = K[x] \setminus \lbrace 0 \rbrace$.
Combining Corollary \ref{cor2loc} and item 1 of Proposition \ref{primeideals}, we get
$$ R(D)_{\p_x} = R(D)[x] = R(S^{-1}D). $$
Set $T=S^{-1}D$. By Lemma \ref{lemfgalg}, $T$ is a finitely generated $K(x)$-algebra. We want to prove that $\dim(T) < \dim(D)$ in order to apply the inductive hypothesis on $T$. To get $\dim(T) < \dim(D)$ we need to show that every maximal ideal of $D$ contains a nonzero element of $K[x]$. But this follows by Zariski's Lemma. Indeed, let $\mathfrak{M}$ be a maximal ideal of $D$. The residue field $\frac{D}{\mathfrak{M}}$ is algebraic over $K$. Thus, if $x \in \mathfrak{M}$ then $x \in \mathfrak{M} \cap K[x] $, otherwise the image of $x$ in $\frac{D}{\mathfrak{M}}$ is algebraic over $K$ and therefore there exists a polynomial $f(t) \in K[t]$ such that $f(x) \in \mathfrak{M}$. In particular $f(x) \in \mathfrak{M} \cap K[x]$. Furthermore, $f(x) \neq 0$, otherwise $x $ would be an algebraic element over $K$ and therefore an Egyptian element of $D$ by Lemma \ref{algegyp}, but this would contradict the fact that $\frac{1}{x} \in \m $.

This proves $\dim(T) < \dim(D)$ and we can finally apply the inductive hypothesis to $T$ to get
$$ n = \hgt \q \leq \hgt \p_x = \dim(R(D)_{\p_x}) = \dim(R(T)) \leq \dim(T) < \dim(D)=n.  $$ This yields a contradiction and forces $\dim(R(D)) \leq \dim(D)$.
 \end{proof}

We use now a similar method involving Zariski's Lemma to construct an interesting example.
In \cite[Theorem 2.10]{Eps3}, it is proved that the reciprocal complement of an Euclidean domain is either a field or a DVR. We show here that this same property fails already for the class of PIDs.

We provide indeed an example of a PID $D$ such that $R(D)$ is not integrally closed. 
This proves that the reciprocal complement of a Pr\"{u}fer domain is not always a Pr\"{u}fer domain and that in general $R(D)$ is not isomorphic to an overring of $D$, conversely to what seen in \cite{EGL} in the case of the polynomial ring $D=K[X_1, \ldots, X_n]$.

\begin{example}
\label{examplenotoverring}
\rm Let us start with the polynomial ring $D=K[X,Y]$ in two variables over a 
perfect field $K$. Let $f=  Y^2-X^3$ and consider the multiplicatively closed set $S = K[f] \setminus \lbrace  0 \rbrace$. By Zariski's Lemma (or Hilbert's Nullstellensatz in the case $K$ is algebraically closed), $S$ intersects all the maximal ideals of $D$. 
It follows that the ring $T=S^{-1}D$ has Krull dimension smaller than two. By Corollary \ref{cor2loc}, $R(T)= R(D)[f]$. 

We adopt now the same notation of \cite[Notation 3.1]{EGL}, defining an automorphism $\sigma$ of $ K(X,Y)$ mapping $X \to X^{-1}$ and $Y \to Y^{-1}$, and setting $R^*= \sigma(R(D))$. Then we have $R(T) \cong R^*[\alpha^{-1}]$ where $\alpha= \sigma(\frac{1}{f}) = \frac{X^3Y^2}{X^3-Y^2} \in R^*$. 

In \cite[Notation 5.4, Proposition 5.5]{EGL} it is shown that there exists a valuation overring $V_{2,3}$ of $R^*$ in which $\alpha$ is a unit. This implies the existence of a prime ideal $\p$ of $R^*$ not containing $\alpha$ (recall that $\alpha$ is not a unit in $R^*$). Hence, $R(T) \neq \mathcal{Q}(D) $ and $T$ is not Egyptian. This shows that $\dim(T)=1$ 
and it is a PID (it is a one-dimensional localization of a Noetherian UFD). 
Also recall that by \cite[Theorem 4.4]{EGL}, $\dim(R(D))=2$, hence $\hgt(\p)=1$ and $\p = \sigma(\p_{f})$ is the unique nonzero prime not containing $\alpha$.

It remains to prove that $R(T)$ is not integrally closed. By \cite[Proof of Theorem 7.2]{EGL}, we have 
$$ D[\alpha] \cong \frac{K[X,Y,Z]}{(-fZ - X^3Y^2)}. $$
Moreover, the ideal $(X,Y)D[\alpha]$ is a prime ideal of height one.

The integral closure of $D[\alpha]$ is $D[\alpha][\beta]$ where $\beta = \frac{X^5Y}{X^3-Y^2}$. The proof of this fact is postponed to Lemma \ref{exintclos}.
Furthermore, $\beta^2, \beta X, \beta Y \in (X,Y)D[\alpha]$ and therefore the integral closure of $D[\alpha]_{(X,Y)}$ is local (and one-dimensional Noetherian), hence a DVR. Call $V$ such a DVR.  

Again by \cite[Proof of Theorem 7.2]{EGL}, we have $$D[\alpha]_{(X,Y)}\subseteq R^*[\alpha^{-1}] = R^*_{\p} \subseteq V.$$ 
Thus $V$ is the unique DVR containing $R^*_{\p} $ and we must have $V=V_{2,3}.$ 
If $v$ is the valuation associated to $V$, then by definition of $V_{2,3}$ (see \cite[Notation 5.4, proof of Theorem 5.8]{EGL}), $v(\beta)=1$ and for every $\phi \in \p$, $v(\phi) \geq 2$. It follows that $\beta \in V \setminus R^*_{\p}$ and $R^*_{\p} = R^*[\alpha^{-1}] \cong R(T)$ is not integrally closed.
\end{example}

\begin{lemma}
\label{exintclos}
Let $D, f, \alpha, \beta $ be defined as in Example \ref{examplenotoverring}. Then $D[\alpha, \beta]$ is the integral closure of $D[\alpha]$.
\end{lemma}

\begin{proof}
It can be easily shown (see \cite[Proof of Theorem 5.8]{EGL}) that $\beta$ is integral over $D[\alpha]$. Hence $D[\alpha, \beta]$ is an integral domain of Krull dimension two. 

Consider the ideal $I= (-fZ - X^3Y^2, YW-ZX^2)$ in the polynomial ring $K[X,Y,Z,W]$. The ring $D[\alpha, \beta]$ can be presented as a quotient of the ring $A= K[X,Y,Z,W]/I, $ by mapping $Z \to \alpha$, $W \to \beta$. The ideal $I$ is generated by a regular sequence, hence $A$ is Cohen-Macaulay of dimension 2. The Jacobian matrix of $I$ is
$$ \bmatrix 
3X^2(Z-Y^2) & -2Y(Z+X^3) & X^3-Y^2 &  0  \\  
-2XZ & W & -X^2 &  Y \\
 \endbmatrix. $$
An easy computation shows that the ideal generated by the maximal minors of this matrix has height $\geq 2$.
Applying the Jacobian criterion \cite[Theorem 18.15]{Eis} to $A$, we get that $A$ is a direct product of normal domains (at this point we require $K$ to be a perfect field). Since $D[\alpha, \beta]$ is (isomorphic to) a quotient of $A$ and the two rings have the same dimension, we must have that $D[\alpha, \beta]$ is a quotient of $A$ at a minimal prime. Therefore $D[\alpha, \beta]$ is isomorphic to one of the direct factors of $A$. It follows that $D[\alpha, \beta]$ is integrally closed.
\end{proof}

\section{Semigroup algebras}

In this section we study the reciprocal complements of semigroup algebras. Many results obtained in \cite{EGL} in the special case of polynomial rings in several variables over a field are generalized to this setting. By \cite[Theorem 2.6]{Eps1}, we know that semigroup algebras over a field are never Egyptian (except if the semigroup is a group). Therefore, they form an interesting class of integral domains for which the reciprocal complement is not trivial and worth to be studied. \\

Let $G \subseteq \oplus_{i = 1}^N \mathbb{R}$ be a totally ordered abelian group of rank $N$, written in additive notation. 
 Denote by $\boldsymbol{0}$ the unit element of $G$ and by $G_{\geq \boldsymbol{0}} $ the submonoid of elements of $G$ larger than or equal to $\boldsymbol{0}$. 

Let $S$ be a subsemigroup of $G_{\geq \boldsymbol{0}} $ containing $\boldsymbol{0}$ and
let $K$ be any field. We consider semigroup algebras of the form $D= K[S]= K[X^s \, | \, s \in S]$. Denote by $\m_S$ the maximal ideal of $D$ generated by all the elements $X^s$ for $s > \boldsymbol{0}$. The quotient field of $D$ is the field $K(S)= K(X^s \, | \, s \in S) $. In general, if not otherwise specified, we suppose $K(S) = K(G)$ (i.e. $G$ is the group generated by $S$). 
Given an element $\phi = \sum_{i=1}^k X^{s_i} \in K[S]$ we will refer to the set $ \lbrace s_1, \ldots, s_k \rbrace$ as the support of $\phi$.

We want to study the reciprocal complement $R(D)$. For simplicity of notation we denote this ring also by $R_S$.
It is not much restrictive to study semigroup algebras over fields rather than over rings. Indeed, in light of Theorem \ref{reductiontoKalgebra}, to study the reciprocal complement it is sufficient to restrict to $K$-algebras such that the only units are the elements of $K$. For this reason we will also usually suppose $S \subseteq G_{\geq \boldsymbol 0}$.


Polynomial rings in finitely many variables over $K$ are special examples of this construction. 
Here we use a similar method as the one used in \cite{EGL}, defining an automorphism of $K(S)$ and studying the (isomorphic) image of $R_S.$ This technique has the advantage of studying the reciprocal complement by identifying it with a proper overring of $D$ (this works for semigroup algebras but it is not true in general as seen in Example \ref{examplenotoverring}). 

Define the automorphism $\sigma: K(G_{\geq \boldsymbol{0}}) \to K(G_{\geq \boldsymbol{0}})$ such that $\sigma(u)=u$ for $u \in K$ and $\sigma(X^g)= X^{-g}$ for every $g \in G_{\geq \boldsymbol{0}}$. 
Set $R^*_S=\sigma(R_S)$. 

Observe that, if $T$ is another semigroup such that $S \subseteq T \subseteq G_{\geq \boldsymbol{0}}$, then clearly $R_S \subseteq R_T$ and, by definition of $\sigma$ we also have $ R^*_S \subseteq R^*_T. $

Notice also that, for $f = \sum_{i=1}^m u_i X^{s_i}$ with $u_i \in K$, $s_i \in S$, we have 
\begin{equation}
\label{eqsigma}
\sigma \left( \frac{1}{f} \right) = \frac{1}{\sum_{i=1}^m u_i X^{-s_i}} = \frac{X^s}{\sum_{i=1}^m u_i X^{s-s_i}},
\end{equation}
where $s$ can be chosen to be any element of $S$ such that $s \geq s_i$ (or even $s - s_i \in S$) for every $i=1, \ldots, n$. 
Elements satisfying this property always exists since $s_1+\ldots + s_n \in S$. For $f= X^s$, we simply get $\sigma \left( \frac{1}{f} \right)= X^s$.
The ring $R^*_S$ is generated over $K$ by all the elements $ \sigma \left( \frac{1}{f} \right) $ for $f \in D$. Every element of $R^*_S$ can be written as a finite sum of such elements. 

The following results generalize the ones already proved in the case where $D$ is a polynomial ring.

\begin{lemma}
\label{lemma1}
The ring $R^*_S$ contains the localization $D_{\m_S}$.
\end{lemma}

\begin{proof}
Pick $f \in D \setminus \m_S$. Then we can write $f = u + \sum_{i=1}^m u_i X^{s_i}$ for $u, u_i $ nonzero elements of $K$. We have $f = \sigma(u + \sum_{i=1}^m u_i X^{-s_i}) \in R^*_S$ and by \cite[Lemma 2.3]{EGL},  $u + \sum_{i=1}^m u_i X^{-s_i}$ is a unit in $R_S$. Therefore $f $ is a unit in $R^*_S$. It follows that $D_{\m_S} \subseteq R^*_S$. 
\end{proof}



The case where $S = G_{\geq \boldsymbol 0}$ is particularly interesting. If $S = \N$ we already know from \cite[Example 2.14]{Eps3} that $D = K[X]$ and $R(D)=K[X^{-1}]_{(X^{-1})}$ is a valuation ring. This is true more in general for every $ G_{\geq \boldsymbol 0} $.
According to the terminology of \cite[Definition 2.1]{Eps3}, the rings $R_{G_{\geq \boldsymbol 0}}$ are \it Bonaccian, \rm meaning that the reciprocal complement is a valuation ring.

\begin{thm}
\label{totalordervaluation}
Suppose $S = G_{\geq  \lbrace \boldsymbol 0 \rbrace}$. Then $R^*_S = D_{\m_S}$ is a valuation domain. 
\end{thm}

\begin{proof}
By Lemma \ref{lemma1}, $R^*_S \supseteq D_{\m_S}$. Pick $f \in S \setminus K$ and write $\sigma(\frac{1}{f})= \frac{X^s}{\sum_{i=1}^m u_i X^{s-s_i}}$ as in equation (\ref{eqsigma}). Since $\leq$ is a total order we can assume $s_1 < \ldots < s_m$ and for every $i$ we have $s_m  \geq  s_i $. Thus, for every $i$, $s_m - s_i \in S$ and we can set $s = s_m$. It follows that $\sigma(\frac{1}{f}) \in D_{\m_S}$ and hence $R^*_S = D_{\m_S}$. It is well-known that $D_{\m_S}$ is a valuation domain setting $v(X^s) = s$ for every $s$.
\end{proof} 

In the following, we denote the maximal ideal of $R^*_{G_{\geq \boldsymbol 0}}$ by $\mathfrak{M}$. Hence, we have $R^*_{G_{\geq \boldsymbol 0}}= K[G_{\geq \boldsymbol 0}]_{\mathfrak{M}}.$

\begin{rem}
\label{remarksem}
By the previous results, if $S \subseteq G_{\geq \boldsymbol 0}$, we have inclusions $$ D_{\m_S} \subseteq R^*_S \subseteq K[G_{\geq \boldsymbol 0}]_{\mathfrak{M}}. $$
By Theorem \ref{islocal!!}, the maximal ideal of $R^*_{S}$ is generated by all the elements $ \sigma \left( \frac{1}{f} \right) $ for $f \in \m_S$. In particular we can obtain the maximal ideal of $R^*_{S}$ as $\mathfrak{M} \cap R^*_S,$ and the only Egyptian elements of $D$ are the nonzero elements of $K$.
\end{rem}

\begin{rem}
\label{remarkbonaccian}
In Example \ref{examplenotoverring} we described a Pr\"{u}fer domain (even a PID) that is not Bonaccian. Instead the ring $D=K[Z, Y, \frac{Y}{Z}, \frac{Y}{Z^2}, \ldots]$ is an example of a Bonaccian domain that is not Pr\"{u}fer. This ring can be seen as a semigroup algebra taking the semigroup $\N \oplus \N$ ordered lexicographically and setting $Z=X^{(0,1)}$, $Y=X^{(1,0)}$. By Theorem \ref{totalordervaluation}, $R(D) \cong D_{(Z)}$ is a valuation domain and $D$ is Bonaccian. But the localization $$D_{(Z+1, Y, \frac{Y}{Z}, \frac{Y}{Z^2}, \ldots)} = K[Z,Y]_{(Z+1, Y)}$$ is not a Pr\"{u}fer domain. Thus neither $D$ is a Pr\"{u}fer domain.
The classification of all Bonaccian domains is still an open problem.
\end{rem}

The fact that, for semigroup algebras, the reciprocal complement is isomorphic to an overring allows to get a precise characterization of when it is Noetherian. 

  \begin{prop}
\label{noethsem}
The following conditions are equivalent:
\begin{enumerate}
\item[(1)] $R(K[S])$ is Noetherian.
\item[(2)] $K[S]$ is Noetherian and one dimensional.
\end{enumerate}
\end{prop}

\begin{proof}
Since $ R^*_S$ is an overring of $K[S]$ isomorphic to $R(K[S])$, one implication follows by Krull-Akizuki's Theorem \cite[Theorem 11.7]{matsumura}.
Suppose then $R(K[S]) \cong R^*_S$ to be Noetherian. 
By Corollary \ref{noeth1}, it has to be one dimensional. In this case Krull-Akizuki's Theorem implies that $K[G_{\geq \boldsymbol 0}]_{\mathfrak{M}}$ is Noetherian. Since $K[G_{\geq \boldsymbol 0}]_{\mathfrak{M}}$ is a valuation domain, it must be a DVR. It follows that $G_{\geq \boldsymbol 0} \cong \N$ and $S$ is isomorphic to a numerical semigroup. Thus, $K[S]$ is Noetherian and one dimensional. 
\end{proof}

\begin{rem}
For arbitrary integral domain $D$ we know that the same result of Proposition \ref{noethsem} is not true. However we ask whether this holds for any $K$-algebra $D$ whose only Egyptian elements are the nonzero element of $K$. 
\end{rem}

We focus now on the case where $K[G_{\geq \boldsymbol 0}]$ is the integral closure of $D=K[S]$. In this case we can give an explicit description of $R^*_S$ as the localization of a semigroup algebra at the maximal ideal generated by the homogeneous elements. We recall that the integral closure of semigroup algebras is a root closure (see \cite[Section 12]{gilmersem}, \cite{rootclosure}) and $K[G_{\geq \boldsymbol 0}]$ is the integral closure of $D$ if and only if for every $g \in G_{\geq \boldsymbol 0}$ there exists $e >0$ such that $eg \in S$. This case includes the case where $S$ is a numerical semigroup. 
For the next main result, we suppose the field $K$ to be of characteristic zero.
We define now a semigroup $S'$ such that $S \subseteq S' \subseteq G_{\geq \boldsymbol 0}$.

For every element $s \in S$ and every $n \geq 1$ we define the set 
$$ S_{n}(s)= \left\lbrace ns - \sum_{i=1}^{n-1} s_i \, | \, s_1, \ldots, s_{n-1} \in S, \, s_1, \ldots, s_{n-1} < s \right\rbrace. $$
We define $S'$ to be the semigroup generated by all the elements in the sets $S_{n}(s)$ for $s \in S$ and $n \geq 1$. Observing that $ S_{1}(s) = \lbrace s \rbrace $ and that if $t \in S_{n}(s)$, then $t \geq s$, we obtain that $S' \supseteq S$ and, if $S$ has a minimal nonzero element $g_1$, then $g_1$ is also the minimal nonzero element of $S'$.

\begin{lemma}
\label{technical}
Let $K[S]$ be a semigroup algebra over the field $K$ and let $\phi \in K[S] \setminus K$.
Fix a nonzero unit $u \in K$. Then for every $e \geq 1$, $$ (\phi+u)(\phi^2+u^2) \cdots (\phi^{2^{e-1}}+u^{2^{e-1}}) = \sum_{j=0}^{2^e-1} \phi^j u^{2^e-1-j}. $$ 
\end{lemma}
\begin{proof}
The proof is an easy induction on $e$.
\end{proof}



\begin{thm}
\label{numerical}
Let $K$ be a field of characteristic zero.
Let $D=K[S]$ be a semigroup algebra and suppose that $K[G_{\geq \boldsymbol 0}]$ is the integral closure of $D$.
Let $S'$ be defined as above. 
Then $R^*_S = K[S']_{\m_{S'}}$.
\end{thm}

\begin{proof}
We already know by Lemma \ref{lemma1} that $R^*_S$ contains $K[S]_{\m_S}$. We first prove the containment $R^*_S \subseteq K[S']_{\m_{S'}}$.

For $f \in K[S] \setminus K$, write $$\sigma(\frac{1}{f})= \frac{X^s}{\sum_{i=1}^m u_i X^{s-s_i}}$$ as in equation (\ref{eqsigma}). If $m=1$, we clearly have $f=X^s \in K[S']$.
Thus, we can assume without loss of generality $m \geq 2$, $s_1 < \ldots < s_m$, $u_m=-1$ and $s=s_m$. 

We have then $s_m - s_i \geq \boldsymbol 0$ and,
if $s_m - s_i \in S$ for every $i$, we get $\sigma(\frac{1}{f}) \in K[S]_{\m_S} \subseteq K[S']_{\m_{S'}}$. 
Otherwise write the denominator of $ \sigma(\frac{1}{f}) $ as $\phi - 1$ where $\phi = \sum_{i=1}^{m-1} u_i X^{s-s_i}$.
Hence, for every $e \geq 1$ we can write $$ \sigma\left( \frac{1}{f} \right) = \frac{X^s (\phi+1)(\phi^2+1) \cdots (\phi^{2^{e-1}}+1)}{(\phi^{2^{e}}-1)}. $$
By assumption on $ S$ and $G_{\geq \boldsymbol 0}$ we can find $e$ large enough such that $\phi^{2^{e}} \in K[S]$ and therefore $\phi^{2^{e}}-1$ is a unit in $K[S]_{\m_S} \subseteq K[S']_{\m_{S'}}$. 

To show that $\sigma(\frac{1}{f}) \in K[S']_{\m_{S'}}$ we need to show that $\psi= X^s (\phi+1)(\phi^2+1) \cdots (\phi^{2^{e-1}}+1) \in K[S']_{\m_{S'}}$. By Lemma \ref{technical}, $\psi = X^s \sum_{j =0}^{2^e-1} \phi^j$. Using the definition of $\phi$, we can express $\psi$ as a linear combination over $K$ of elements of the form $X^{ns - \sum_{j=1}^{n-1} t_j}$ with $t_j \in \lbrace s_1, \ldots, s_{m-1} \rbrace$ and $n \leq e$. Such elements are all in $K[S']$ by definition of $S'$. This proves the inclusion $R^*_S \subseteq K[S']_{\m_{S'}}$ and also that the maximal ideal of $R^*_S$ is contained in the maximal ideal of $K[S']_{\m_{S'}}$.

 To conclude it suffices to show that $K[S'] \subseteq R^*_S$. Thus suppose $S \neq S'$ and
 pick $g \in S' \setminus S$. Write $g = ns - \sum_{i=1}^{n-1} s_i$ for $s, s_i \in S$, $s_i \leq s$. These elements $s_i$ may not be all distinct, thus we can rename them and write $g = ns - \sum_{i=1}^{k} n_it_i$ with $t_i = s_j$ for some $j$, $n_i \in \N$, $\sum_{i=1}^{k} n_i = n-1 \geq 1$, and the $t_i$'s are all distinct. For some $v_1, \ldots, v_k, u \in K \setminus \lbrace 0 \rbrace$
 set $f= \sum_{i=1}^k v_i X^{t_i} - u X^{s}$ and $\phi = \sum_{i=1}^k v_i X^{s - t_i}$. Observe that, using Lemma \ref{technical} as before, for every $e \geq 1$ we have
 $$ \sigma \left( \frac{1}{f} \right) = \frac{X^s}{\phi-u} =\frac{X^s (\phi+u)(\phi^2+u^2) \cdots (\phi^{2^{e-1}}+u^{2^{e-1}})}{(\phi^{2^{e}}-u^{2^e})}= \frac{X^s \sum_{j =0}^{2^e-1} u^{2^e-1-j} \phi^j}{(\phi^{2^{e}}-u^{2^e})}. $$ We can choose $e$ large enough such that $2^e \geq n$ and $\phi^{2^{e}} \in K[S]$ (we can choose such $e$ in an uniform way, independently on the choice of coefficients $v_1, \ldots, v_k$). Under this assumption we clearly have that $ \phi^{2^{e}}-1 $ is a unit in $K[S]_{\m_S} \subseteq R^*_S$. Hence, we can clear the denominator to get the element $$\psi=  \sum_{j =0}^{2^e-1} u^{2^e-1-j} (X^s \phi^j) \in R^*_S.$$
 Since this procedure can be applied for arbitrary choices of $u$ and $K$ is infinite, solving a linear system over $K$ by a Vandermonde determinant argument shows that $X^s\phi^j \in R^*_S$ for every $j=1, \ldots, 2^e-1$. Observe that by construction all the $a \in G$ that are in the support of $X^s\phi^j$ are elements of $S'$.
 
 But now we can change the coefficients $v_1, \ldots, v_k$ in the field $K$, defining $k$ new elements 
 $\phi_r = \sum_{i=1}^k v_{i,r} X^{s - t_i}$ for $r=1, \ldots, k$.
 Then we can apply the same argument as above to obtain $X^s\phi_r^j \in R^*_S$ for every $j=1, \ldots, 2^e-1$ and $r=1, \ldots, k$.
 Notice that of course each $\phi_r$ has the same support of $ \phi$ in $S'$. 
 
 The choice of coefficients $v_{i,r}$ in the infinite field $K$ can be made generic enough that, solving another linear system over $K$ we get $X^a \in R^*_S$ for every $a$ in the support of $X^s\phi^j$, for every $j=1, \ldots, 2^e-1$.  
 
 Finally, it is easy to show that $g$ is in the support of $X^s \phi^{n-1}$ (since $K$ has characteristic zero, one can choose the coefficients $v_1, \ldots, v_k$ to be positive integers to avoid cancellations). By our assumption on $e$, we have $n-1 \leq 2^e-1 $ and therefore $X^g \in R^*_S$.
\end{proof}

  
  \begin{example}
\label{examplesnumerical}
Let $S$ be a numerical semigroup generated by positive integers $g_1, \ldots, g_t$ such that $g_1 < g_2 < \ldots < g_t$. If every $g \geq g_2$ is in $S$, then $S = S'$ and therefore $R^*_S = K[S]_{\m_S} $ exactly as in the case $S = \N$. Instead, if for instance $S$ is generated by $4,7,9$ we have that $ 10 = 2\cdot7 -4 \in S' \setminus S$. In this case $S'$ is generated by $4,7,9,10$.
\end{example}

The next topic we consider is the computation of the Krull dimension of $R_S$. It is well-know that the dimension of $K[S]$ is $N$, where $N$ is the rank of the group $G$ (see \cite[Section 21]{gilmersem}). Furthermore, it is also known that semigroup algebras over a field are Jaffard domain (see \cite[Corollary 1.18]{jaffard}). This means that the Krull dimension of every overring of $R_S$ is at most $N$. Applying the map $\sigma$ we can identify $R_S$ with an overring of $K[S]$ and obtain for free that the dimension of $R_S$ is at most $N$. This shows that semigroup algebras are another class of integral domain satisfying Conjecture \ref{conj1}.

To compute the dimension more precisely we introduce a family of subsemigroups of $S$.   Using that $G $ has rank $N$, we can write every element of $G$ as $(g_1, \ldots, g_N)$ with each $g_i$ in the projection of $G$ on the $i$-th component. We can choose a total order on the direct summands of $G$ and see the total order of $G$ as a lexicographic order on its components. We say that $ (g_1, \ldots, g_N) < (\gamma_1, \ldots, \gamma_N)$ if $g_j= \gamma_j $ for $j \leq i$ and $g_i < \gamma_i$.

For $i=1, \ldots, N$ we set 
$$ G_i = \lbrace (g_1, \ldots, g_N) \, |  \, g_j = 0 \mbox{ for } j < i \mbox{ and } g_i > 0 \rbrace \subseteq G_{\geq \boldsymbol 0}. $$
  Set then $S_i = G_i \cap S$ and $S_0 = \lbrace \boldsymbol 0 \rbrace$. 
  Observe that for every $j,k$, $S_j + S_k \subseteq S_{\tiny \min \lbrace j,k \rbrace}$.
  
  Set 
  $\Sigma_i = \lbrace X^s \, |  \, s \in \bigcup_{j=i}^N S_j \rbrace $. Then 
  $\Sigma_i$ is a multiplicatively closed set of $D$ 
  and $K[\Sigma_i]$ is a semigroup algebra, with associated monoid $S_0 \cup \Sigma_i$.
  We prove the following:

\begin{thm}
\label{dimth1}
Let $D=K[S]$ be a semigroup algebra of Krull dimension $N$ and suppose that all the sets $S_i= G_i \cap S$ are non-empty. Then $\dim (R(D)) = N$.
\end{thm}

\begin{proof}
From the fact that $D$ is a Jaffard domain and $R_S$ is isomorphic to an overring of $D$ we have $\dim (R_S) \leq N$.
Also recall that $D$ is not Egyptian and therefore $\dim (R_S) \geq 1$. If $N=1$, the result follows immediately. Suppose then $N > 1$ and 
consider the chain of overrings
$$ R_S \subsetneq R_S[\Sigma_N] \subseteq R_S[\Sigma_{N-1}] \subseteq \ldots \subseteq R_S[\Sigma_{2}] \subseteq R_S[\Sigma_{1}]. $$ By Lemma \ref{localizations} and Theorem \ref{islocal!!}, $R_S[\Sigma_i]= R(\Sigma_i^{-1}K[S])$ is a local domain and it is obtained from $R_S$ inverting a set of elements. Hence, there exists a prime ideal $\q_i$ such that $R_S[\Sigma_i] = (R_S)_{\q_i}$. 

We want to prove that for every $i$, the containment $R_S[\Sigma_i] \subseteq R_S[\Sigma_{i-1}]$ is strict. For this we claim that 
\begin{equation}
\label{claim}
R_S[\Sigma_i]= R(\Sigma_i^{-1}K[S]) = R \left(K(\Sigma_i)\left[ \bigcup_{j<i}S_j \right] \right),
\end{equation}
where $K(\Sigma_i)$ denotes the quotient field of $K[\Sigma_i]$ and $\bigcup_{0 \leq j<i}S_j$ is clearly a monoid. 

The inclusion $R(\Sigma_i^{-1}K[S]) \subseteq R(K(\Sigma_i)[\bigcup_{j<i}S_j])$ follows from the inclusion $\Sigma_i^{-1}K[S] \subseteq K(\Sigma_i)[\bigcup_{j<i}S_j]$. For the opposite inclusion we use the following argument: 
the ring $K[\Sigma_i, \Sigma_i^{-1}]$ is a $K$-algebra generated by Egyptian elements, therefore it is Egyptian and $$K(\Sigma_i) = R(K[\Sigma_i, \Sigma_i^{-1}]) \subseteq R(\Sigma_i^{-1}K[S]).$$ An element of $K(\Sigma_i)[\bigcup_{j<i}S_j]$ can be written as $\frac{f}{\epsilon}$ with $f \in K[S]$ and $\epsilon \in K(\Sigma_i)$. This implies that its reciprocal is in $R(\Sigma_i^{-1}K[S])$.

Recall now that by hypothesis the set $S_{i-1}$ is not empty.
By Remark \ref{remarksem} and by the equality (\ref{claim}), it follows that every element $X^g$ for $g \in S_{i-1}$ is not Egyptian in the semigroup algebra $K(\Sigma_i)[\bigcup_{j<i}S_j]$. Hence, $X^g \not \in R_S[\Sigma_i]$. This shows $R_S[\Sigma_i] \subsetneq R_S[\Sigma_{i-1}]$.

 Furthermore, notice that $(R_S)_{\q_1}= R_S[\Sigma_1]= R(\Sigma_1^{-1}K[S]) = R(K[X^s, X^{-s}, \, s \in S])$ is equal to the quotient field of $D$ since $K[X^s, X^{-s}, \, s \in S]$ is Egyptian.
 
 To conclude, call $\m$ the maximal ideal of $R_S$.
From what proved above we have a chain of prime ideals
$$ \m \supsetneq \q_{N} \supsetneq \q_{N-1} \supsetneq \ldots \supsetneq \q_2 \supsetneq \q_1 = (0). $$ This proves that $R_S$ has dimension $N$.
\end{proof}

\begin{rem}
\label{remdim}
From the fact that $D=K[S]$ is a Jaffard domain and from the proof of Theorem \ref{dimth1}, it follows that $N-t \leq \dim (R(D)) \leq N $ where $t $ is the number of sets $S_i$ that are empty.
\end{rem}

Certain cases where some of the sets $S_i$ are empty will be treated in the next section using $D+\m$ constructions.

 
 \section{Reciprocal complements of $D+\m$ constructions}
 
In this section we describe the reciprocal complements of certain $D+\m$ constructions in order to construct examples of integral domains $D$ such that the difference $\dim(D)-\dim(R(D))$ is arbitrarily large.
 
 Given an integral domain $D$ and an ideal $I \subseteq D$, we denote by $R(I)$ the ideal of $R(D)$ generated by the elements of the form $\frac{1}{x}$ for $x \in I$.
 
 If $\p$ is a prime ideal of $D$, the ideal $R(\p)$ is not necessarily a prime ideal. For instance $(X)$ is a prime ideal of $D=K[X,Y]$ but the ideal generated by $ \frac{1}{X}$ in $R(K[X,Y])$ is not prime since 
 $$ \frac{1}{Y}\left[ \frac{1}{X} - \frac{1}{X+Y} \right]=  \frac{1}{Y}  \frac{Y}{X(X+Y)} = \frac{1}{X}\frac{1}{X+Y}, $$ and both $\frac{1}{Y}$ and $ \frac{1}{X+Y} $ are not in the ideal generated by $\frac{1}{X}$ (to see this observe that $ \frac{X}{Y} \not \in R(D) $ by \cite[Example 2.9]{Eps3} and, by a linear change of coordinates in $D$, also $ \frac{X}{X+Y} \not \in R(D) $).
 However $R(\p)$ is prime in the following situation:
 
\begin{lemma}
\label{rpprimo}
Let $D$ be an integral domain of the form $D = \mathcal{E} + \p$, where $\mathcal{E}$ is the subring of $D$ consisting of all the Egyptian elements of $D$ together with zero and $\p$ is a prime ideal such that $\p \cap \mathcal{E} = (0)$. Then $R(\p)$ is the maximal ideal of $R(D)$ and 
$$ R(D)=\mathcal{Q}(\mathcal{E})+R(\p), $$ where $ \mathcal{Q}(\mathcal{E}) $ denotes the quotient field of $\mathcal{E}$.
\end{lemma}

\begin{proof}
The ideal $R(\p)$ is a proper ideal since by assumptions the elements of $\p$ are not Egyptian. Pick $x \in D \setminus \p$. Then $x = e+z$ with $e $ an Egyptian element and  $z \in \p$. If $z = 0$, clearly $\frac{1}{e}$ is a unit in $R(D)$. If $z \neq 0$ we observe that by Theorem \ref{reductiontoKalgebra}, $R(D)= R(E^{-1}D)$ where $E = \mathcal{E} \setminus \lbrace 0 \rbrace$. Localizing to the ring $E^{-1}D$ we can use Theorem \ref{islocal!!} 
to show that $ \frac{ 1}{e} + \frac{1}{z}$ is a unit in $R(D)$. Since $\frac{1}{z} \in R(\p)$, it follows that
$$ \frac{1}{x}= \frac{1}{e+z} = \frac{1}{z} \frac{1}{e} \frac{1}{\frac{1}{z} + \frac{1}{e}} \in R(\p). $$ 
Hence, $R(\p)$ is the maximal ideal of $R(D)$.

Since taking rings of reciprocals preserves inclusion, we clearly have $ \mathcal{Q}(\mathcal{E}) \subseteq R(D) $ and therefore $ \mathcal{Q}(\mathcal{E})+R(\p) \subseteq R(D) $. Every element of $R(D)$ is a sum of reciprocals of elements of $D$. By what proved above, all the reciprocals of elements of $D$ are either in $ \mathcal{Q}(\mathcal{E}) $ or in $R(\p)$. It follows that $ R(D)=\mathcal{Q}(\mathcal{E})+R(\p). $
\end{proof}

Of course the above result holds in the case where $\mathcal{E}$  is a field and $\p$ is a maximal ideal of $D$. We prove 
now a more general theorem describing the reciprocal complements of certain $D+\m$ and $K+\m$ constructions. For a detailed background on this kind of ring constructions we refer to \cite{gilmer}.

\begin{thm}
\label{d+m}
Let $T$ be an integral domain 
and let $\mathcal{E}$ be the subring of the Egyptian elements of $T$ together with zero. Let $\p$ be a prime ideal of $T$ such that $\mathcal{E} \cap \p = (0)$. 
Let $B$ be a subring of $\mathcal{E}$ and set $D = B + \p$. Then 
$$ R(D)= R(B)+R(\p) $$ 
and $R(\p)$ is a prime ideal of $R(D)$.
\end{thm}

\begin{proof}
The containment $R(B)+R(\p) \subseteq R(D) $ follows immediately from the definitions. Let then $x = b+z \in D$ where $b \in B$ and $z \in \p$. We need to show that $ \frac{1}{x} \in R(B)+R(\p) $. 

If either $b$ or $z$ is zero this is obvious. 
Thus assume both elements to be nonzero, observe that $b$ is Egyptian in $T$ and argue as in the proof of Lemma \ref{rpprimo} to obtain $$ \frac{1}{x} = \frac{1}{b+z} = \frac{1}{z} \frac{1}{b} \frac{1}{\frac{1}{z} + \frac{1}{b}} \in \left( \frac{1}{z} \right) R(T) \subseteq R(\p). $$
This proves the inclusion $R(D) \subseteq R(B)+R(\p) $.
Finally, since $R(B)$ is an integral domain and $R(B) \cap R(\p)= \lbrace 0 \rbrace$, $R(\p)$ is a prime ideal of $R(D)$.
\end{proof}

We apply Theorem \ref{d+m} to the following examples of semigroups rings. In this context we prefer to express the semigroup algebras with a notation more similar to the case of polynomial rings.

\begin{example}
\label{ex1}
Let $X,Y,Z$ be indeterminates over a field $K$.
\begin{enumerate}
\item Let $D= K[YX^k, k \in \Z]$. As a semigroup ring we can see $D$ as $K[S]$ where $S$ is the submonoid of $\Z \times \Z$ consisting of $(0,0) $ and all the elements $(a,b)$ with $a >0$, $b \in \Z$. According to the notation introduced in the previous section, the set $S_2$ is empty in this case. Let $\p$ be the maximal ideal of $D$ generated by all the elements $YX^k$ for $k \in \Z$. Set $$T = K[X, X^{-1}, Y] = K[X, X^{-1}] + \p.$$ The element $Y$ is not Egyptian in $T$ since it is not Egyptian in the overring $K(X)[Y]$. Hence, all the elements of $\p$ are not Egyptian (factors of Egyptian elements are Egyptian). 
Recall that by \cite[Example 2.14]{Eps3} $R(K[X])= K[X^{-1}]_{(X^{-1})}$.

By Corollary \ref{cor2loc} (or by \cite[Proposition 7]{GLO}) we obtain that $K[X, X^{-1}]$ is an Egyptian domain, and therefore it is the subring of the Egyptian elements of $T$. 
We can first apply Lemma \ref{rpprimo} to $T$ to get $ R(T)= K(X) + R(\p) $. 
On the other hand, by
Theorem \ref{reductiontoKalgebra} and \cite[Example 2.14]{Eps3}, we get $$R(T)= R(K(X)[Y]) = K(X)[Y^{-1}]_{(Y^{-1})}.$$
It follows that $\p = Y^{-1}K(X)[Y^{-1}]_{(Y^{-1})}.$

Applying now Theorem \ref{d+m} choosing $B=K$,
we get
$$ R(D) = R(K)+R(\p) = K + Y^{-1} K(X)[Y^{-1}]_{(Y^{-1})}. $$ 
Applying the automorphism $\sigma$ of $K(X,Y)$ defined in the previous section (and originally in \cite{EGL}), we get $R(D) \cong K + YK(X)[Y]_{(Y)}$, which is one of the most classical examples of $K + \m$ constructions. It is well-know that in this case $\dim(R(D))=1$, while $\dim(D)=2$. This example is generalized to arbitrary dimension in Theorem \ref{dimth2}.
\item Let $D= K[YX^k, k \geq 0] = K + \p$, where $\p = (Y)K[X,Y]$. Set $T= K[X]+\p = K[X,Y]$. 
By Theorem \ref{d+m}, applied choosing $B=K$, we get $$R(D) = R(K)+R(\p)= K+ Y^{-1}R(K[X,Y]).$$
Notice that in this case $R(\p)$ is the maximal ideal of $R(D)$, but is not a prime ideal of $R(T)$.
\item Let $D= K[X, \frac{Y}{X^k}, \frac{Z}{X^k}, k \geq 0]$. In this case let $\p $ be the ideal generated by $\frac{Y}{X^k}, \frac{Z}{X^k}$. for $k \geq 0$. We apply Theorem \ref{d+m} using $T= D[X^{-1}]= K[X,X^{-1}] + \p$ and $B=K[X]$. This gives 
$$ R(D) = R(K[X]+\p) = R(K[X])+R(\p) = K[X^{-1}]_{(X^{-1})}+ \mathcal{M}, $$ where $\mathcal{M}$ is the maximal ideal of $R(K(X)[Y,Z])$.
\item An "opposite" example with respect to the preceding one is obtained by setting $D = K[Y,Z, \frac{X}{Y^k}, \frac{X}{Z^k}, k \geq 0]$ and letting $\p  $ be the ideal generated by $\frac{X}{Y^k}, \frac{X}{Z^k}$ for $ k \geq 0$. Define $$T = D[Y^{-1}, Z^{-1}]= K[Y,Z,Y^{-1}, Z^{-1}] + \p.$$ Then, applying Theorem \ref{d+m} for the choice $B=K[Y,Z]$, we obtain
$$ R(D) = R(K[Y,Z])+ X^{-1}K(Y,Z)[X^{-1}]_{(X^{-1})}. $$
\end{enumerate} 
\end{example}


\begin{rem}
\label{ex2}
The assumption $B \subseteq  \mathcal{E} $ is necessary in Theorem \ref{d+m}. Indeed, take $D=T=K[X,Y]= K[Y]+\p$ where $\p=(X)D$.
In this case $B=K[Y]$ is not an Egyptian subring of $T$ and we have $R(D) \supsetneq R(K[Y])+R(\p)$, since, as we observed at the beginning of this section, $R(\p) $ is not prime in $R(D)$.
\end{rem}

We consider now a family of semigroup algebras for which the Krull dimension of the reciprocal complement coincides with the lower bound given in Remark \ref{remdim}. As a consequence we can show that for every pair of integers $n \geq m \geq 0$ there exists an integral domain $D$ of Krull dimension $n$ such that the reciprocal complement has dimension $m$. This family arises as a generalization of item 1 of Example \ref{ex1} to an arbitrary number of variables.

We recall that if $T =L + \m$ is a local domain where $L$ is a field, and $K \subseteq L$ is a smaller field, then $\dim(K+\m)= \dim(T)$ (see for instance \cite{gilmer} or \cite{GH}).


\begin{thm}
\label{dimth2}
Let $K$ be a field and consider indeterminates $Y_1, \ldots, Y_m, X_{m+1}, \ldots, X_n$ for $n > m \geq 1$. Let $D$ be the semigroup algebra generated over $K$ by all the elements of the form $X_i^kY_j$ with $j= 1, \ldots, m$, $i=m+1, \ldots, n$ and $k \in \Z$. Then $$  R(D) = K+ \mathfrak{M} $$ where $\mathfrak{M}$ is the maximal ideal of $R(K(X_{m+1}, \ldots, X_n)[Y_1, \ldots, Y_m])$.
In particular, $$\dim (D) = n > \dim (R(D))=m.$$
\end{thm}

\begin{proof}
Let $\p$ be the maximal ideal of $D$ generated by all the elements $X_i^kY_j$. Set 
$$T = D[X_{m+1}, \ldots, X_n, X_{m+1}^{-1},\ldots, X_n^{-1}] = K[X_{m+1}, \ldots, X_n, X_{m+1}^{-1},\ldots, X_n^{-1}] + \p.$$ 
All the elements $Y_1, \ldots, Y_m$ are not Egyptian in $T$ since they are not Egyptian in the overring $K(X_{m+1}, \ldots, X_n)[Y_1, \ldots, Y_m]$. It follows that all the elements of $\p$ are not Egyptian in $T$ and $K[X_{m+1}, \ldots, X_n, X_{m+1}^{-1},\ldots, X_n^{-1}]$ is the subring of Egyptian elements of $T$. 
Hence, by Lemma \ref{rpprimo} $$R(T)= R(K(X_{m+1}, \ldots, X_n)[Y_1, \ldots, Y_m]) = K(X_{m+1}, \ldots, X_n) + \mathfrak{M}$$ where $\mathfrak{M} = R(\p)$.
By Theorem \ref{d+m}, 
we get
$$ R(D) = R(K+\p)= R(K)+R(\p) = K + \mathfrak{M}. $$
 Since $D$ is a semigroup algebra over $K$ defined by a monoid of rank $n$, we get $\dim(D)=n$. By what recalled above on $K+\m$ constructions, we have $$\dim(R(D))= \dim(R(K(X_{m+1}, \ldots, X_n)[Y_1, \ldots, Y_m])) = m, $$
 where the last equality follows by \cite[Theorem 4.4]{EGL} or by Theorem \ref{dimth1}.
\end{proof}

\begin{cor}
\label{dimn&m}
Given any pair of integers $n \geq m \geq 0$, there exists an integral domain $D$ of Krull dimension $n$ such that its reciprocal complement $R(D)$ has Krull dimension $m$.
\end{cor}

\begin{proof}
If $m=0$, choose $D$ to be any Egyptian domain of dimension $n$ (for instance a local domain). If $m=n$, we can choose $D=K[X_1, \ldots, X_n]$ or any semigroup algebra satisfying the assumptions of Theorem \ref{dimth1}. If $n > m > 0$, we can use the integral domains constructed in Theorem \ref{dimth2}.
\end{proof}

 \section{The case when $R(D)$ is a DVR}
 
 In this section we prove a sort of converse of the main result of \cite{Eps3}, stating that the reciprocal complement of a Euclidean domain is either a field or a DVR. Here, we prove that if $D$ is an integral domain such that $R(D)$ is a DVR (not a field), then the localization $E^{-1}D$ has to be isomorphic to a polynomial ring in one variable over a field. The main argument of the proof is the use of the discrete valuation of $R(D)$ to define a Euclidean function on $D$.
 
 We recall that an integral domain $D$ is a Euclidean domain if there exist a function $f: D \setminus \lbrace 0 \rbrace \to \Z$ such that for every nonzero $a,b \in D$:
 \begin{enumerate}
 \item $f(ab)\geq f(a).$
 \item There exist $q,r \in D$ such that $a=bq+r$ and either $r=0$ or $f(r) < f(b)$.
\end{enumerate}  
 
\begin{thm}
 \label{dvrcase}
 Let $D$ be an integral domain and denote by $E$ the set of the Egyptian elements of $D$. Suppose that $D$ is not Egyptian. The following conditions are equivalent:
 \begin{enumerate}
 \item[(1)] $R(D)$ is a DVR.
 \item[(2)] $E^{-1}D$ is a Euclidean domain.
 \item[(3)] $E^{-1}D$ is isomorphic to a polynomial ring in one variable over a field.
\end{enumerate}  
 \end{thm}
 
 \begin{proof}
By Theorem \ref{reductiontoKalgebra} we know that $R(D)=R(E^{-1}D)$. Therefore, we can assume without loss of generality that $D=E^{-1}D$ is a $K$-algebra whose only Egyptian elements are the nonzero elements of the field $K$, and $D \neq K$. The implication  (3) $\to$ (2) is straightforward while the implication (2) $\to$ (1) is the content of \cite[Theorem 2.10]{Eps3}.

Let us prove the implication (1) $\to$ (2). Suppose $R(D)=V$ is a DVR with associated valuation $v$ and denote by $\m$ its maximal ideal. Let $\pi= \sum_{i=1}^c \frac{1}{x_i} \in R(D)$, with $x_i \in D$, be an element such that $v(\pi)=1$. In particular $\pi \in \m$ and also $ \frac{1}{x_i} \in \m $ for every $i$ by item 3 of Proposition \ref{primeideals}. In particular there exists at least one $i$ such that $ v(\frac{1}{x_i})=1 $ and we can change generator for $\m$ and suppose $\pi = \frac{1}{y}$ for some $y \in D \setminus K$.

Define now a function $f: D \setminus \lbrace 0 \rbrace \to \Z$ as $f(d):= v(\frac{1}{d})$. Pick nonzero $a,b \in D$. Since $\frac{1}{b} \in V$, we obtain immediately $$f(ab)= v \left( \frac{1}{ab}\right) = v\left(\frac{1}{a}\right) + v\left(\frac{1}{b}\right) \geq v\left(\frac{1}{a}\right) = f(a). $$
To prove that $D$ is a Euclidean domain we need to find a suitable $q \in D$ such that
$$ f(a-bq) =  v \left( \frac{1}{a-bq}\right) < v\left(\frac{1}{b}\right) = f(b). $$
Set $e_1 = f(a)$ and $e_2 = f(b)$.
If $b$ divides $a$ in $D$ we simply choose $q=\frac{a}{b}$ and $r=0$, while if $e_1 < e_2$ we choose $q=0$ and $r=a$. Thus, we can assume that $b$ does not divide $a$ and $e:= e_1 - e_2 \geq 0 $. These two conditions exclude the cases where one of $a, b$ is a unit in $D$. Indeed, if $b$ is a unit, then it divides $a$, while if $a $ is a unit and $b$ is not, then $\frac{1}{a}$ is a unit in $R(D)$ and $ \frac{1}{b} \in \m $, hence $e_1 = 0 < e_2 $.

Thus $e_1 \geq e_2 >0 $ and, since $ \frac{1}{a}, \frac{1}{b} \in \m = (\frac{1}{y})V$, we can write $a = y^{e_1} \theta_1 $, $b = y^{e_2} \theta_2 $ where $\theta_1, \theta_2$ are units in $R(D)$.
From \cite[Lemma 2.3]{EGL}, Theorem \ref{islocal!!}, and Theorem \ref{reductiontoKalgebra}, we know that each unit in $R(D)$ can be written as a sum of a nonzero element of $K$ and some element in $\m$. But any element of $\m$ is obtained as the product of a power of $ \frac{1}{y}$ and a unit of $V$. Therefore we can write $$a = y^{e_1} \left( u_1 + \sum_{i=1}^{e} \frac{a_i}{y^i} + \alpha \right), \quad b = y^{e_2} \left( u_2 + \sum_{i=1}^{e} \frac{b_i}{y^i} + \beta \right) $$ where $u_1, u_2 \in K \setminus \lbrace 0 \rbrace$, $a_i, b_i \in K$, and $\alpha, \beta \in \m^{e+1}$. Set $q = y^{e} (u + \sum_{i=1}^{e} \frac{c_i}{y^i} )$ where $u, c_1, \ldots, c_{e} \in K$ are elements that we are going to determine in function of $u_1,u_2, a_i, b_i$. By definition it is clear that $q \in K[y] \subseteq D$.
Thus, $a-bq = y^{e_1}\eta $ where 
$$ \eta = u_1 + \sum_{i=1}^{e} \frac{a_i}{y^i} + \alpha - \left( u_2 + \sum_{i=1}^{e} \frac{b_i}{y^i} + \beta \right)\left(u + \sum_{i=1}^{e} \frac{c_i}{y^i} \right).  $$
We clearly have $ v \left( \frac{1}{a-bq}\right) = e_1 - v(\eta) < v\left(\frac{1}{b}\right) $ if and only if $v(\eta)>e_1 -e_2 = e.$
Since $v(\alpha), v(\beta) > e$, it is sufficient to show that for some choice of $u, c_1, \ldots, c_{e} \in K$, the term
$$ \eta' = u_1 + \sum_{i=1}^{e} \frac{a_i}{y^i} - \left( u_2 + \sum_{i=1}^{e} \frac{b_i}{y^i} \right)\left(u + \sum_{i=1}^{e} \frac{c_i}{y^i} \right)\in \m^{e+1}.  $$ We can rewrite
$$ \eta' = u_1 - u_2u + \sum_{i=1}^{e} \frac{1}{y^i} \left[ a_i - ub_i -u_2c_i - \sum_{\scriptstyle j+k = i \atop \scriptstyle  j,k \geq 1} b_j c_k \right] + \gamma, $$ with $\gamma \in \m^{e+1}$.
Hence, we can first set $u:= u_1 u_2^{-1}.$ Then we observe that working inductively for $i=1, \ldots, e$, we can choose $c_i$ in such a way that the coefficient in front of 
$\frac{1}{y^i}$ is zero. In particular,
$$ c_i:= u_2^{-1}\left(a_i - ub_i - \sum_{\scriptstyle j+k = i \atop \scriptstyle  j,k \geq 1} b_j c_k \right) $$ can be determined univocally in function of $a_i, b_j$ and $c_k$ with $k < i$. It follows that, for this choice of coefficients, $\eta' = \gamma \in \m^{e+1}$.
This choice of $q $ is such that $f(a-bq) < f(b)$ and the proof of the implication (1) $\to$ (2) is complete.

Finally, we prove the implication (1) $\to$ (3). Let us adopt the same notation of the proof of (1) $\to$ (2). Then we know that there exists $y \in D$ such that $v(\frac{1}{y})=1.$ Since $y$ is not a unit in $D$, then it is not algebraic over the field $K$. Hence $K[y]$ is isomorphic to a polynomial ring over $K$. We clearly have $K[y] \subseteq D$. 
Pick then $x \in D$. 
By the proof of (1) $\to$ (2), we can consider the Euclidean function $f$ defined on $D$ to  
find $q,r \in D$ such that $x=yq+r$ and either $r=0$ or $f(r) < f(y) = v(\frac{1}{y})=1$. Hence, if $r \neq 0$, we must have $v(\frac{1}{r})=0$ and therefore $r \in D \cap R(D) = K$.
But in the proof of (1) $\to$ (2) it was shown that $q$ can be chosen such that $q \in K[y]$. It follows that $x = yq+r \in K[y]$.
 \end{proof}





  
  
  \section*{Acknowledgements}
  The author is supported by the grants MAESTRO NCN-UMO-2019/34/A/ST1/00263 - Research in Commutative Algebra and Representation
Theory and NAWA POWROTY-PPN/PPO/2018/1/00013/U/00001 – Applications of
Lie algebras to Commutative Algebra.

The author would like to thank Neil Epstein and K. Alan Loper for interesting conversations about the content of this article.

\end{document}